\documentclass[letterpaper,11pt]{article}
\usepackage{amsmath,amssymb,amsthm}
\usepackage{graphicx}
\usepackage{hyperref}
\usepackage[lined,boxed,linesnumbered,commentsnumbered]{algorithm2e}
\usepackage{subfig}

\newtheorem{theorem}{Theorem}
\newtheorem{proposition}[theorem]{Proposition}
\newtheorem{corollary}[theorem]{Corollary}
\newtheorem{observation}[theorem]{Observation}
\newtheorem{example}{Example}

\theoremstyle{definition}
\newtheorem{definition}[theorem]{Definition}

\begin{document}
\title{Zero forcing for inertia sets}
\author{Steve Butler \and Jason Grout \and H.\ Tracy Hall}
\date{\empty}
\maketitle

\begin{abstract}
Zero forcing is a combinatorial game played on a graph with a goal of turning all of the vertices of the graph black while having to use as few ``unforced'' moves as possible.  This leads to a parameter known as the zero forcing number which can be used to give an upper bound for the maximum nullity of a matrix associated with the graph.

We introduce a new variation on the zero forcing game which can be used to give an upper bound for the maximum nullity of a matrix associated with a graph that has $q$ negative eigenvalues.  This gives some limits to the number of positive eigenvalues that such a graph can have and so can be used to form lower bounds for the inertia set of a graph.
\end{abstract}

\section{Introduction}
Given a graph $G$ we can associate a collection of matrices ${\cal S}(G)$, such that each matrix $A\in {\cal S}(G)$ is Hermitian and the off-diagonal entries are non-zero if and only if the entry corresponds to an edge of the graph; the diagonal entries can be either zero or nonzero.  We can then ask questions about matrices in ${\cal S}(G)$.  One example is the inverse eigenvalue problem, which asks what set of eigenvalues are possible for matrices belonging to ${\cal S}(G)$.  The complete inverse eigenvalue problem has been solved for very few graphs \cite{inverse}.

A relaxation of the inverse eigenvalue problem is to compute the minimum possible rank of matrices in ${\cal S}(G)$, denoted ${\rm mr}(G)$.  Since the diagonal has no restrictions this is equivalent to determining the maximum possible multiplicity of an eigenvalue for matrices in ${\cal S}(G)$, in particular the maximum nullity of a matrix in ${\cal S}(G)$.  (Recall the nullity is the dimension of the eigenspace associated with the eigenvalue $0$.)  Any matrix in ${\cal S}(G)$ gives a lower bound for the nullity of a matrix associated with the graph, but upper bounds must come from a more general argument.

One method of producing upper bounds for the nullity that has proven successful for small graphs and graphs with some degree of structure is the zero forcing number of a graph (see \cite{AIM,AIM2,HCY}).  The zero forcing number, denoted $Z(G)$, is the smallest number of vertices that when colored black will ``force'' all vertices to be colored black using a color change rule which we will discuss in the next section.

A problem lying between the inverse eigenvalue problem and the maximum nullity problem is determining the inertia set of a graph, denoted ${\cal I}(G)$.  This is the collection of all possible pairs $(p,q)$ where $p$ is the number of positive eigenvalues and $q$ the number of negative eigenvalues for $A\in{\cal S}(G)$.  Given $p$ and $q$ the multiplicity of zero as an eigenvalue, i.e., nullity, is $|G|-p-q$.  In particular, an upper bound on the nullity of such a matrix given that it has $q$ negative eigenvalues will produce a lower bound for $p$ and so constrains the inertia set of a graph.

In this paper we will introduce a new combinatorial game played with two players that bounds the nullity of matrices associated with a graph and which have $q$ negative eigenvalues.  This new game is a generalization of previous zero forcing games with an extra parameter $q$ which will show up in a new forcing rule involving the second player.  The minimal cost it will take for the Black player in the game to color all vertices Black will be denoted by $Z_q(G)$.  We will show how $Z_q(G)$ gives an upper bound for the nullity of a matrix $A\in {\cal S}(G)$ with $q$ negative eigenvalues.

For the special case $q\ge n$ the parameter $Z_q(G)=Z(G)$, i.e., the regular zero forcing number, while for $q=0$, the parameter $Z_0(G)=Z_+(G)$, the positive semidefinite zero forcing number (introduced in \cite{AIM2}).

In Section~\ref{sec:Z} we will give a short review of both zero forcing and  positive semidefinite zero forcing, which we will generalize in Section~\ref{sec:Zq} with our new zero forcing parameter $Z_q(G)$.  We then will introduce $\widehat{Z}_q(G)$ in Section~\ref{sec:hat} which will give a further improvement for determining the inertia set of a graph.  In Section~\ref{sec:final} we will show how to use these new parameters to bound the inertia set of a graph.  Finally, in Section~\ref{sec:algo} we will give an algorithm for computing $Z_q(G)$.

We will use the following notation in this paper:  For a subset of vertices, $W\subseteq V$, let $G[W]$ be the induced subgraph of $G$ on the vertices $W$.

\section{Review of zero forcing and semidefinite zero forcing}\label{sec:Z}
Zero forcing can be thought of as a game, played on the graph $G$ by a single player.  The game consists of two simple operations, coloring a vertex black and a color change rule.

\begin{definition}
The \emph{color change rule}\/ for a graph with vertices painted black and white gives that any black vertex with exactly one white vertex as a neighbor (and arbitrarily many black vertices), the unique adjacent white vertex changes color from white to black.  An application of the color change rule is known as \emph{forcing}.
\end{definition}

The color change rule will play an important role in the analysis of various zero forcing games we will discuss.  The motivation behind the definition comes from looking at vectors in the null space of a matrix.  Suppose ${\bf x}$ is in the null space for a matrix $A\in{\cal S}(G)$, i.e., we have $A{\bf x}={\bf 0}$.  Further, when the vertices of the graph are colored black the corresponding entries of ${\bf x}$ are $0$ while the vertices colored white are undetermined.  The color change rule is the observation that we might also have some additional entries of ${\bf x}$ that must also be $0$.  This is because if vertex $i$ is black and its only white neighbor is vertex $j$ then, $(A{\bf x})_i=a_{ij}x_j=0$.  But since $a_{ij}\ne0$ then we must have $x_j=0$, i.e., the vertex in the graph can also be colored black.

We are now ready to give the zero forcing game.

\bigskip\noindent\hfil\boxed{\parbox{0.97\textwidth}{
{\bf Zero Forcing Game} -- All the vertices of the graph $G$ are initially colored white and there is one player, known as Black, who has a collection of tokens.  Black will repeatedly apply one of the following two operations until all vertices are colored black:
\begin{itemize}
\item Black can change any vertex from white to black, at the cost of a single token.
\item Black can apply the color change rule on the entire graph.  This operation does not cost Black a token.
\end{itemize}
}}\hfil\bigskip

The minimal number of tokens that Black must use in a given strategy to change all of the vertices from white to black is the zero forcing number $Z(G)$.  Because of the non-adaptive nature of the game, Black can choose to first only consider the option of using tokens to change vertices from white to black and then apply the second option of using the color change rule.  In this setting the set of vertices that Black initially spends his tokens on are known as a {\em zero forcing set}\/ and $Z(G)$ is then the minimal size of a zero forcing set.  This is the way that zero forcing is usually defined and introduced (i.e., as a set rather than a strategy).

By the color change rule it follows that any vector in the null space of $A\in{\cal S}(G)$ which is $0$ on a zero forcing set of $G$ is the ${\bf 0}$ vector.  On the other hand, by considering the dimensions of various subspaces we have the following.

\begin{observation}\label{obs:dim}
If the nullity of a matrix is more than $k$, then for any $k$ specified entries there is a nonzero vector ${\bf x}$ in the null space which will vanish at those specified entries.
\end{observation}

We immediately have the following.

\begin{proposition}[AIM \cite{AIM}]\label{prop:AIM}
For any $A\in{\cal S}(G)$ the dimension of the nullity is bounded above by the size of any zero forcing set, in particular by $Z(G)$.
\end{proposition}
\begin{proof}
Suppose not, then by the above observation there would exist a nonzero vector in the null space which is $0$ on all the vertices corresponding to a minimally sized zero forcing set.  But this is impossible since the only vector in the null space which is $0$ on a zero forcing set is ${\bf 0}$.
\end{proof}

A modification of zero forcing was considered when the matrices in ${\cal S}(G)$ were further restricted to require that the matrices be positive semidefinite.  Since this imposes additional relationships on entries in the matrix it is possible to modify the game to give Black more power in forcing vertices to be black.

\bigskip\noindent\hfil\boxed{\parbox{0.97\textwidth}{
{\bf Semidefinite Zero Forcing Game} --  All the vertices of the graph $G$ are initially colored white and there is one player, known as Black, who has a collection of tokens.  Black will repeatedly apply one of the following three options until all vertices are colored black:
\begin{itemize}
\item Black can change any vertex from white to black, at the cost of a single token.
\item Black can apply the color change rule on the entire graph.  This operation does not cost Black a token.
\item Let the vertices currently colored black be denoted by $B$, and $W_1,W_2,\ldots,W_k$ be the vertex sets of the connected components of $G[V\setminus B]$.  Black can  apply the color change rule on $G[B\cup W_i]$ for some $1\le i\le k$.  This operation does not cost Black a token.
\end{itemize}
}}\hfil\bigskip

In particular we have the same option to spend tokens and for applying the color change rule on the whole graph.  We also have an option allowing us to apply the color change rule on a {\em smaller} part of the graph --- that is, we may be allowed to ignore some of the white neighbors of a black vertex, leaving ``exactly one'' white neighbor under consideration, which can then be forced black.  The minimal number of tokens that Black must use to change all of the vertices from white to black in this game is denoted $Z_+(G)$.  As before, Black can elect to initially only spend tokens and then apply either of the options using the color change rule for the remainder of the game, and so the literature discusses positive semidefinite forcing sets and not positive semidefinite forcing strategies.  In this setting $Z_+(G)$ is the size of the smallest possible such set.

\begin{theorem}\label{thm:Z+}
The nullity of $A\in {\cal S}(G)$ when $A$ is positive semidefinite is at most $Z_+(G)$.
\end{theorem}

This is a special case of Theorem~\ref{thm:Zq} and we will omit the proof here.  The original proof can be found in \cite{AIM2}.

\section{Zero forcing with $q$ negative eigenvalues}\label{sec:Zq}
The semidefinite zero forcing number gives an indication of how to generalize zero forcing.  Namely, we give Black the possibility of working with a smaller graph.  This leads us to the general $Z_q$-forcing game.

\bigskip\noindent\hfil\boxed{\parbox{0.97\textwidth}{
{\bf $Z_q$-Forcing Game} -- All the vertices of the graph $G$ are initially  colored white and there are two players, known as Black (who has a collection of tokens) and White.  Black will repeatedly apply one of the following three options until all vertices are colored black.
\begin{itemize}
\item Black can change any vertex from white to black, at the cost of a single token.
\item Black can apply the color change rule on the entire graph $G$.  This operations does not cost Black a token.
\item Let the vertices currently colored black be denoted by $B$, and $W_1,\ldots,W_k$ be the vertex sets of the connected components of $G[V\setminus B]$.  Black selects {\em at least}\/ $q+1$ of the $W_i$, announcing the set $S\subseteq\{W_1,\ldots,W_k\}$ (with $k\ge q+1$) to White.  Then White will select a nonempty subset $T\subseteq S$, say $T=\{W_{i_1},\ldots,W_{i_\ell}\}$ (with $\ell\ge 1$) and announces it back to Black.  Black can apply the color change rule on $G[B\cup W_{i_1}\cup\cdots\cup W_{i_\ell}]$.  This operation does not cost Black a token.
\end{itemize}
}}\hfil\bigskip

In this game White has the goal of making it as costly for Black as possible, or in other words, White is trying to delay the process of all the vertices from black.  The minimal number of tokens that Black must use to change all of the vertices from white to black, regardless of the play of White, is the zero forcing number for matrices with $q$ negative eigenvalues, denoted $Z_q(G)$.

\begin{example}
Consider the graph shown in Figure~\ref{fig:exa} on $10$ vertices.  A simple case analysis shows that $Z_1(G)>2$. (Black must always start by spending a token, and if the graph does not have a leaf then in the $Z_1$-forcing game Black must spend at least two tokens before anything can occur.  But by symmetry one can easily show that White can prevent Black from forcing more than a single vertex.)
\begin{figure}[hftb]
\centering
\subfloat[]{\includegraphics{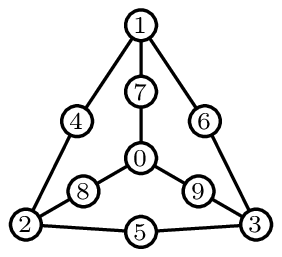}\label{fig:exa}}\hfil
\subfloat[]{\includegraphics{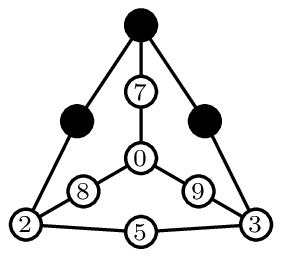}\label{fig:exb}}\hfil
\subfloat[]{\includegraphics{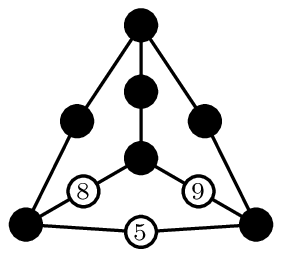}\label{fig:exd}}
\caption{An example of  a graph with $Z_1(G)=3$.}
\end{figure}
We now show that $Z_1(G)=3$.  In this case Black spends three tokens to color the vertices $1$, $4$ and $6$ black, as shown in Figure~\ref{fig:exb}.  Then by forcing on the whole graph the vertices $2$, $7$ and $3$ become black, as shown in Figure~\ref{fig:exd}.  There are now three connected components when the black vertices are removed, Black picks any two of them and declares them to White.  Whatever White returns can be forced and then using the color change rule on the entire graph will finish changing the remaining white vertices to black.
\end{example}

Unlike zero forcing and positive semidefinite zero forcing, the strategy for Black may no longer be to spend all the tokens up front.  In other words there are graphs where Black will vary the choice of where to spend tokens depending on the response of White.  So we do not have $Z_q$-forcing sets but $Z_q$-forcing strategies.

\begin{example}
Consider the graph shown in Figure~\ref{fig:ex2a} on $9$ vertices.  It will follow from Corollary~\ref{cor:Zq} along with the known properties of this graph that $Z_1(G)\ge 4$.  We will show a strategy for Black that uses $4$ tokens to make all the vertices of the graph black.
\begin{figure}[hftb]
\centering
\subfloat[]{\includegraphics{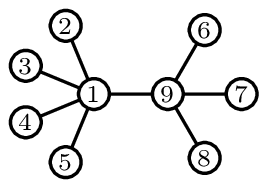}\label{fig:ex2a}}\hfil
\subfloat[]{\includegraphics{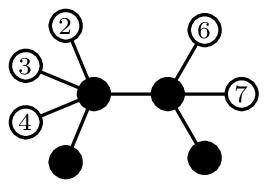}\label{fig:ex2b}}\hfil
\subfloat[]{\includegraphics{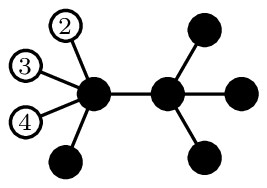}\label{fig:ex2c}}
\caption{An example of  a graph with $Z_1(G)=4$.}
\end{figure}
First, Black spends two tokens to color the vertices $5$ and $8$ black.  Since these are pendant vertices these then force the vertices $1$ and $9$ to black by using the color change rule on the whole graph.  Black now picks a vertex from $\{1,2,3\}$ and a vertex from $\{6,7\}$ and declares it to White, whatever is returned is then forced and this continues until one of the two sets has been completely colored black.  In the worst case scenario for Black, White will have only returned the vertices from $\{6,7\}$ in which case the graph is as shown in Figure~\ref{fig:ex2c}.  At this point Black can spend at most two tokens to get all but one of the remaining vertices to become black and the last vertex will then be switched to black by the color change rule.
\end{example}

If in the previous example Black had chosen to initially spend four tokens before using any free forcing, then at least two of $\{2,3,4,5\}$ or two of $\{6,7,8\}$ would not be black.  At this point White can now always block Black from coloring the two remaining vertices without spending a token (i.e., the color change rule on the whole graph cannot color the vertices since they share a common neighbor and White can always keep those two vertices white when Black declares some connected components).

This also shows that White should not always return all the subsets Black has declared even though intuitively it would seem that the more white vertices there are, the more difficult it should be for Black to apply the color change rule.  (The important point is that the intuition holds only when the white components are ``close''.)

For each graph $G$ and each value $q$ we have a parameter $Z_q(G)$.  These are nicely ordered in the following sense.

\begin{proposition}\label{prop:order}
For any graph we have
\[
Z_+(G)=Z_0(G)\le Z_1(G)\le Z_2(G)\le\cdots\le Z(G).
\]
\end{proposition}
\begin{proof}
For $q=0$, Black can select $1$ component, and then White has no choice but to return that component, we can thus remove White from the game and reduce to the game for semidefinite zero forcing.

Suppose that $t\le s$.  Black can now use the strategy for the $Z_s$-forcing game that uses at most $Z_s(G)$ tokens in the $Z_t$-forcing game and force all the vertices to be black.  It follows that $Z_t(G)\le Z_s(G)$.

Finally, Black can choose to color all the vertices of a zero forcing set black, at a cost of $Z(G)$ tokens and then change all the remaining vertices  to black using the color change rule.  This shows that $Z_q(G)\le Z(G)$ for all $q$.  Alternatively, the role of White only comes into play if there are at least $q+2$ connected components in $G[V\setminus B]$.  When the number of connected components is always below $q$ then we reduce to the original zero forcing game.  In other words, we can think of the zero forcing number of a graph as $Z_\infty(G)$.
\end{proof}

On the other hand, the parameter $Z_q(G)$ can behave differently than the zero forcing and semidefinite zero forcing parameters.  For example, if $G\sqcup H$ is the disjoint union of $G$ and $H$, then it is easy to show that $Z(G\sqcup H)=Z(G)+Z(H)$ and $Z_+(G\sqcup H)=Z_+(G)+Z_+(H)$, i.e., since to force the union of graphs we must force each graph individually.  However, we have $Z_q(G\sqcup H)\le Z_q(G)+Z_q(H)$, and we usually have a strict inequality.  For example, consider $K_{1,p}\sqcup K_{1,q}$, with $p,q\ge 2$.  It is easy to show that $Z_1(K_{1,p})=p-1$ and $Z_1(K_{1,q})=q-1$.  But using a strategy similar to the one given in the second example we have 
\[
Z_1(K_{1,p}\sqcup K_{1,q})=\max\{p,q\}-1<p+q-2=Z_1(K_{1,p})+Z_1(K_{1,q}).
\]

As before, the parameter $Z_q(G)$ gives information about the maximum nullity for some subset of the matrices in ${\cal S}(G)$.

\begin{theorem}\label{thm:Zq}
The nullity of $A\in {\cal S}(G)$ when $A$ has exactly $q$ negative eigenvalues is at most $Z_q(G)$.
\end{theorem}

We remark the proof we will give for Theorem~\ref{thm:Zq} also works to show that ``exactly $q$'' can be relaxed to ``at most $q$''.  Before we begin the proof we first need to introduce isotropic subspaces.

\begin{definition}
An \emph{isotropic subspace}\/ of a matrix is a subspace where $\mathbf{x}^*A\mathbf{x}=0$ for all $\mathbf{x}$ in the subspace.
\end{definition}

\begin{theorem}{(\cite[Theorem 1.5]{inner})}\label{thm:iso}
The maximum possible dimension of an isotropic subspace for a Hermitian matrix $A$ is $n-p-q+\min\{p,q\}$ where $p$ and $q$ are the number (counting multiplicity) of positive and negative eigenvalues respectively.
\end{theorem}

\begin{corollary}\label{cor:iso}
For a Hermitian matrix $A$, let $R$ be an isotropic subspace of dimension more than $\min\{p,q\}$ where $p$ and $q$ are the number (counting multiplicity) of positive and negative eigenvalues respectively.  Then $R$ contains a vector in the null space.
\end{corollary}
\begin{proof}
Any of the $n-p-q$ eigenvectors associated with eigenvalue $0$ can be in the isotropic subspace, beyond this, by Theorem~\ref{thm:iso} there are at most $\min\{p,q\}$ vectors making the isotropic subspace.  The result now follows by dimension arguments.
\end{proof}

%The second tool we will need is a framework for describing the game.
%
%\begin{definition}
%A \emph{$Z_q$-haven} for the graph $G$ is a subset $U$ of the vertices so that if $U$ are the only currently white vertices in the graph, there are no free forces for the Black player.
%\end{definition}
%
%In particular $Z_q$-havens are points in the game where in order for Black to color any more vertices, they would have to pay.
%
%\begin{definition}
%A \emph{level $0$ $Z_q$-haven} is any haven including the empty set.  For each $k\ge 1$ a \emph{level $k$ $Z_q$-haven} is a nonempty haven $U$ such that for each $v\in U$, when black does free forcing on the graph where $U-v$ are the only white vertices then White can have Black stop on a level $k-1$ haven.
%\end{definition}
%
%Note by this definition that a haven might be associated with it.  Generally, we are interested in the highest level that a haven is associated with.  This is because these levels are used to indicate how much Black has to spend starting from their current haven to get all vertices colored black, i.e., the haven corresponding to the empty set which is a level $0$ $Z_q$-haven.  In particular we have the following observation which follows from the definitions.
%
%\begin{observation}
%The haven $V$ (i.e., the initially all white vertex configuration at the start of the game) is a level $Z_q(G)$ $Z_q$-haven but is not a level $(Z_q(G)+1)$ $Z_q$-haven.
%\end{observation}

\begin{proof}[Proof of Theorem~\ref{thm:Zq}]
Given a matrix $A$ with $q$ negative eigenvalues and multiplicity of $m$ for the eigenvalue $0$, we will use the matrix to produce a strategy for White in the $Z_q$-forcing game.  This strategy can be used to show that starting from $V$ that black will have to spend at least $m$ times to force all vertices to be black, showing that $m\le Z_q(G)$, as desired.

The strategy will be as follows: After every time Black spends a token then White will look at the space $N$ which is the intersection of the null space and the span of vectors with nonzero entries only for the white vertices.  Let $T\subseteq V$ be the support of $N$, i.e., the vertices for which some vector in $N$ is nonzero for that vertex.  Then by general position there is some null vector $\mathbf{x}\in N$ whose support is $T$.  We now show that White can keep black from using the color change rule to switch the color for any vertex in $T$; or in other words, White can protect the support of $N$.

Recall from the discussion following the color change rule in Section~\ref{sec:Z} that the only time a vertex will change from white to black by the rule is if the entry in the null space vector ${\bf x}$ is $0$.  So when applying the color change rule to the entire graph we cannot change the color for anything in the support of ${\bf x}$, i.e., anything in the support of $T$.

So now consider the option where White is involved.  Suppose that $B$ are the vertices colored black and $W_1,W_2,\ldots,W_k$ are the vertices of the maximally connected components of $G[V\setminus B]$.  Then by appropriate relabeling we can assume that
\[
A=\left(
\begin{array}{c|c|c|c|c|c}
A_1&O&O&\cdots&O&B_1^*\\ \hline
O&A_2&O&\cdots&O&B_2^*\\ \hline
O&O&A_3&\cdots&O&B_3^*\\ \hline
\vdots&\vdots&\vdots&\ddots&\vdots&\vdots\\ \hline
O&O&O&\cdots&A_k&B_k^*\\ \hline
B_1&B_2&B_3&\cdots&B_k&C
\end{array}
\right)
\qquad
\mbox{and}\qquad
{\bf x}=
\left(
\begin{array}{c}
{\bf x}_1\\
{\bf x}_2\\
{\bf x}_3\\
\vdots\\
{\bf x}_k\\
{\bf 0}
\end{array}
\right)\,.
\]
Where $A_i$ is the submatrix on $G[W_i]$.  Then for any index $i$ using that $A{\bf x}={\bf 0}$ we have $A_i{\bf x}_i={\bf 0}$.  Let $\widehat{{\bf x}}_i$ be the vector ${\bf x}$ restricted to ${\bf x}_i$ (all the other terms are zeroed out).

If ${\bf y}=\sum a_i\widehat{{\bf x}}_i$, we have
\[
{\bf y}^*A{\bf y}=
{\bf y}^*\left(\begin{array}{c}\vdots\\a_iA_i{\bf x}_i\\\vdots\\\sum a_iB_i{\bf x}_i\end{array}\right)=\left(\begin{array}{cccc} \cdots &a_i{\bf x}_i&\cdots&{\bf 0}\end{array}\right)\left(\begin{array}{c}\vdots\\{\bf 0}\\\vdots\\\sum a_iB_i{\bf x}_i\end{array}\right)=0.
\]
This shows the span of the vectors $\widehat{{\bf x}}_1,\widehat{{\bf x}}_2,\ldots,\widehat{{\bf x}}_k$ form an isotropic subspace of $A$.  Black has selected at least $q+1>\min\{p,q\}$ of the $W_i$, and each $W_i$ is associated with $\widehat{{\bf x}}_i$.

If any of the $\widehat{{\bf x}}_i=\mathbf{0}$ for the corresponding components that Black gave to White, then White will return that single component to Black since then any forcing on that component will not impact the support of $\mathbf{x}$, i.e., the support of $T$.

Now suppose none of the $\widehat{{\bf x}}_i$ are $\mathbf{0}$ for the components Black gives to White.  Then by Corollary~\ref{cor:iso} there is a nontrivial null vector in the subspace spanned by the $\widehat{{\bf x}}_i$.  In particular, there is some null vector ${\bf z}=\sum b_i\widehat{{\bf x}}_i$.  White returns all of the $W_i$ components for which $b_i\ne 0$.

By the same argument as above, any application of the color change rule that Black is able to apply in this returned set cannot impact a nonzero entry of ${\bf z}$.  But the nonzero entries of ${\bf z}$ are precisely the nonzero entries of ${\bf x}$ restricted to the sets that were returned.  So as before Black cannot change the color for a vertex corresponding to a nonzero entry of ${\bf x}$, and so cannot force a vertex in $T$.

In both operations of applying the color change rule we see that Black cannot change the color of a vertex in $T$.  In particular, the only way that Black has to reduce the size of $T$ is to spend a token.  For each token spent this adds one constraint to the set of null vectors that still have support on the white vertices, i.e., it reduces the dimension of the new $N$ by at most one.  Starting with the set of all vertices we have that the dimension of the null space is $m$.  Every time that Black spends the dimension of the null space vectors that are still in the support of the white vertices reduces by at most one, and the process cannot stop until this dimension is $0$ (i.e., the only vector when all vertices are black is ${\bf 0}$).  Therefore Black has to spend a token at least $m$ times.
\end{proof}

The proof shows that a matrix can give a strategy for White to play that forces Black to spend at least the size of the maximum nullity.  However, White does not need to base the strategy off of a matrix and so in general this bound will not be tight.

\section{The parameter $\widehat{Z}_q(G)$}\label{sec:hat}
The parameter $Z_q(G)$ works well for small graphs.  For example on all but one tree up through $10$ vertices we can use $Z_q(G)$ to get tight bounds for the nullity of a matrix associated with the tree which has $q$ negative eigenvalues.  The one exception is the tree $T$ shown in Figure~\ref{fig:BFtree}, sometimes referred to in the literature as the Barioli-Fallat Tree.

\begin{figure}[hftb]
\centering
\includegraphics{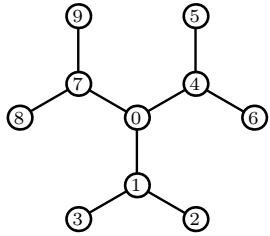}
\caption{The Barioli-Fallat Tree}
\label{fig:BFtree}
\end{figure}

For this tree the maximum nullity for a matrix with $1$ negative eigenvalue is $2$ (a method that computes all possible numbers of positive and negative eigenvalues for a matrix associated with a tree is given in \cite{cut}).  However, we have $Z_1(T)=3$.  It is easy to see that a strategy for Black is to spend tokens on $2,5,8$ and then force to get $1,4,7$; Black then hands White the sets $\{3\}$ and $\{6\}$; whatever is returned is forced and forcing on the whole graph gives the rest.  On the other hand if Black only spends two tokens then simple case analysis will check that one of the pairs $\{2,3\}$, $\{5,6\}$ or $\{7,8\}$ will both be white.  At this point White can now protect that pair and so Black cannot get all the vertices to be black with just $2$ tokens.

Recently an improvement for zero forcing was introduced by Barioli et al.\ \cite{zhat} and is denoted by $\widehat{Z}(G)$.  This can be extended to give an improved $\widehat{Z}_q(G)$, which correctly gives the correct bounds for the Barioli-Fallat tree and improved bounds for many other graphs.

The approach is to introduce auxiliary graphs $\widehat{G}$ which is the same as $G$ except each vertex is placed into one of two groups: looped (which can force itself when it is white with all black neighbors), and unlooped (which if incident to a single white vertex can force that vertex even if itself is white).  The game is unchanged except for a modified color change rule.

\begin{definition}
The \emph{color change rule for a graph with loops and unloops} and vertices painted black and white gives that any vertex with exactly one white vertex as a neighbor (and arbitrarily many black vertices), the unique adjacent white vertex changes color from white to black.
\end{definition}

In particular, we have that white vertices are allowed to force with the convention that a vertex with a loop has itself as a neighbor and a vertex with an unloop does not have itself as a neighbor.  We define 
\[
\widehat{Z}_q(G)=\max_{\widehat{G}\in{\cal G}}Z_q(\widehat{G})
\]
where the maximum runs over the set of all possible auxiliary graphs ${\cal G}$.

The idea behind this is that we also give restrictions to the diagonal entries of the matrix in ${\cal S}(G)$.  In other words we are specifying the zero/nonzero pattern on the diagonal by putting loops at nonzero entries on the diagonal and unloops at zero entries on the diagonal.  We are taking the maximum so that we handle every possible combination of the diagonal.

In theory this is computationally prohibitive because we have to run over all possible ways to loop/unloop the vertices.  However in practice we only need to find a set of situations which cover all possibilities.

Consider again the Barioli-Fallat Tree and let us try to compute $\widehat{Z}_1(T)$.  Suppose that there was a loop at $2$.  Now consider the following strategy for Black.  Black spends at $6$ and $8$; forcing then changes $4$ and $7$ to black; Black now hands White $\{5\}$ and $\{9\}$; whatever is returned is forced and forcing on the entire graph gets us to only the vertices $2$ and $3$ as white; $2$ is looped so it now forces itself to black; forcing then gives $3$.  In this case Black is able to color all the vertices black using only two tokens.  By symmetry the same strategy works if there is any loop at a leaf.

Now consider a strategy for Black when all of the leaves are unlooped.  Without spending Black has that $1,4,7$ are black by the unlooped leaves;  Black hands White $\{2\}$ and $\{6\}$ and whatever is returned is forced, without loss of generality let us suppose that only $\{2\}$ is returned; Black now hands White $\{6\}$ and $\{8\}$ and whatever is returned is forced, without loss of generality let us suppose that only $\{6\}$ is returned; Black spends on $0$ and $8$; the remaining vertices are then forced.

All possibilities of being looped and unlooped falls into one of these two cases and in each case Black only needed to spend at most $2$, showing that $\widehat{Z}_1(T)=2$ (equality follows from what is known about the possible eigenvalues of $T$).

Since Black does not have to use the looped and unlooped vertices, i.e., only lets black vertices force, we have that $\widehat{Z}_q(G)\le Z_q(G)$.  So $\widehat{Z}_q(G)$ is a better bound.  For example for trees the parameter $\widehat{Z}_q(G)$ correctly gives the right nullity for all but one tree up through $16$ vertices.   The one exception is a graph on $16$ vertices related to the Barioli-Fallat Tree shown in Figure~\ref{fig:BFtree16}.  This demonstrates that there is still room for improvement in bounding the nullity of matrices associated with a graph.

\begin{figure}[htb]
\centering
\includegraphics{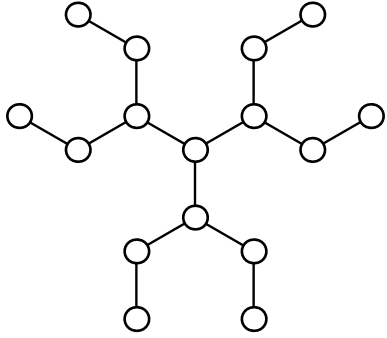}
\caption{The Extended Barioli-Fallat Tree}
\label{fig:BFtree16}
\end{figure}

\section{Finding lower bounds for inertia sets}\label{sec:final}

A primary motivation for the parameters $Z_q(G)$ was to provide lower bounds for inertia sets.  Recall that the inertia set of a graph, ${\cal I}(G)$, is the set of all possible $(p,q)$, where $p$ is the number of positive eigenvalues and $q$ is the number of negative eigenvalues, for matrices in the set ${\cal S}(G)$.  To give upper bounds it suffices to produce matrices, lower bounds are more difficult in that we must show that some matrix is not possible.

Inertia sets satisfy some simple properties.  For example, $A\in {\cal S}(G)$ if and only if $-A\in{\cal S}(G)$, so we have $(p,q)\in{\cal I}(G)$ if and only if $(q,p)\in{\cal I}(G)$.  In other words the inertia sets have symmetry across the line $y=x$.  Further, inertia sets do not have ``holes''.  This is a consequence of the following ``Northeast Lemma'', which says that if a point is in the inertia set then everything above and to the right (up to the dimension constraint) is also in the set.

\begin{proposition}[Barrett, Hall, Loewy\ \cite{cut}]\label{prop:ne}
If $(p,q)\in{\cal I}(G)$ then $(p+s,q+t)\in{\cal I}(G)$ where $s$ and $t$ are nonnegative integers and $(p+s)+(q+t)\le n$.
\end{proposition}

This gives us the ``Southwest Corollary''.

\begin{corollary}
If $(p,q)\not\in{\cal I}(G)$ and $p+q<n$ then $(s,t)\not\in{\cal I}(G)$ where $0\le s\le p$ and $0\le t\le q$.
\end{corollary}

The way we will use $Z_q(G)$ to help give lower bounds for the inertia is by use of the following observation. 

\begin{observation}\label{cor:Zq}
We have $(n-q-Z_q(G)-1,q)\not\in{\cal I}(G)$.
\end{observation}

This follows since $p=|G|-q-m$ where $m$ is the nullity of the matrix, since $Z_q(G)$ is an upper bound on the nullity we can conclude that $p\ge |G|-q-Z_q(G)$.  This shows that if we have $q$ negative eigenvalues then we cannot have $|G|-q-Z_q(G)-1$ positive eigenvalues, i.e., the point is not in the inertia set of the graph.

We can now apply the Southwest Corollary to the following set to give a lower envelope for the points which are not in ${\cal I}(G)$:
\[
\{(n-q-Z_q(G)-1,q),(q,n-q-Z_1(G)-1):0\le q\le n\}.
\]

As an example let us consider Desargues Graph on 20 vertices, shown in Figure~\ref{fig:des}( with some forcing sets marked).  

\begin{figure}[ht]
\centering
\subfloat[$Z$-forcing set]{\includegraphics[scale=0.85]{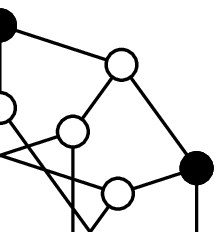}}\hfil
\subfloat[$Z_1$-forcing set]{\includegraphics[scale=0.85]{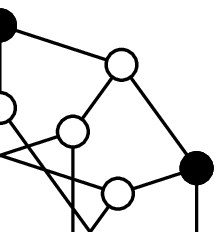}}\hfil
\subfloat[$Z_+$-forcing set]{\includegraphics[scale=0.85]{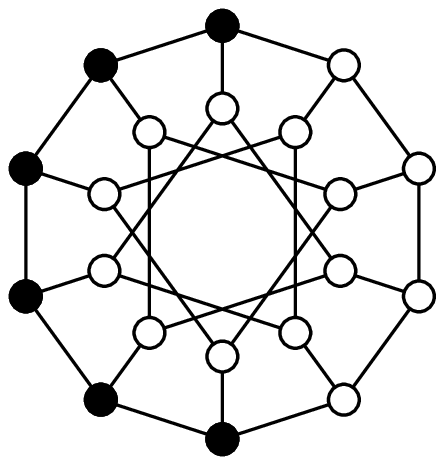}}
\caption{Several forcing sets for various games on Desargues graph.}
\label{fig:des}
\end{figure}

We can use the various parameters we have discussed with differing conditions to find lower bounds for the inertia set of this graph.  These are shown in Figure~\ref{fig:desbnds} and we summarize them below.  We note that since Desargues Graph has a $K_6$ minor that all points with $15\le p+q\le 20$ are in the Inertia set, so we will only look at what happens for $p+q\le 14$.  The circle will indicate a point that is \emph{possibly} in the inertia based on the bounds that a certain tool yields.

\begin{figure}[ht]
\centering
\subfloat[Using $Z$]{\includegraphics[scale=0.6]{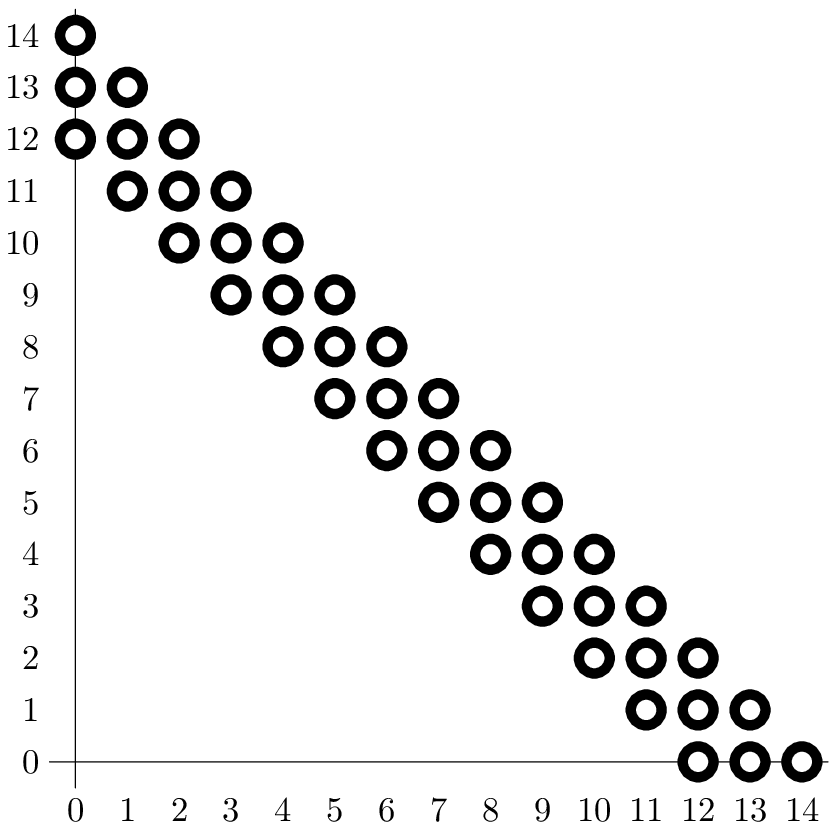}}\hfil
\subfloat[Using $Z$ and $Z_+$]{\includegraphics[scale=0.6]{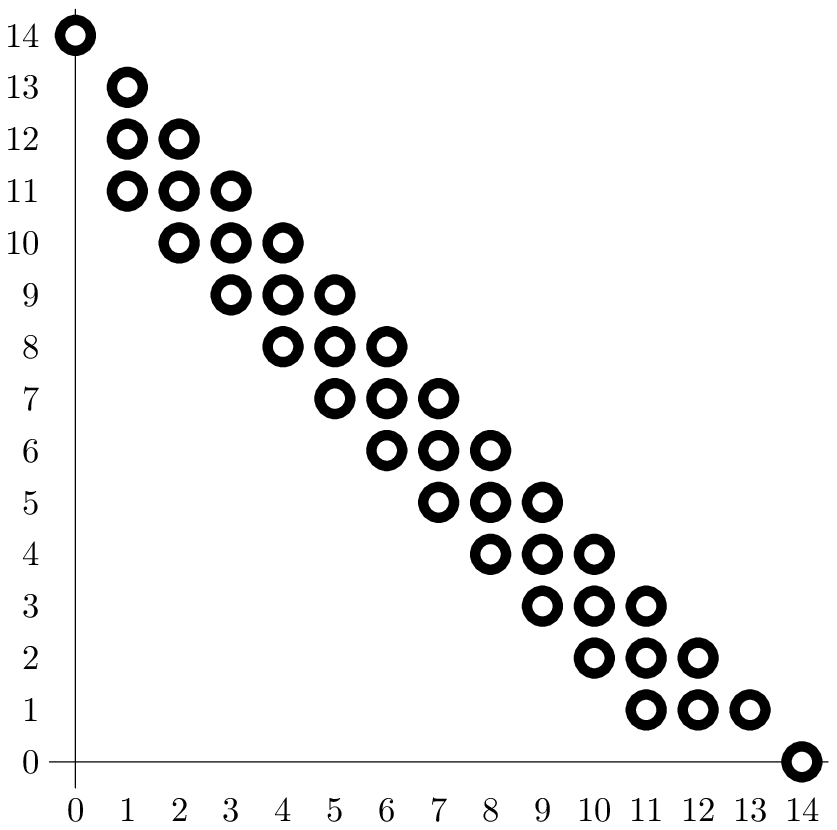}}\hfil
\subfloat[$Z_q$ with spending up front]{\includegraphics[scale=0.6]{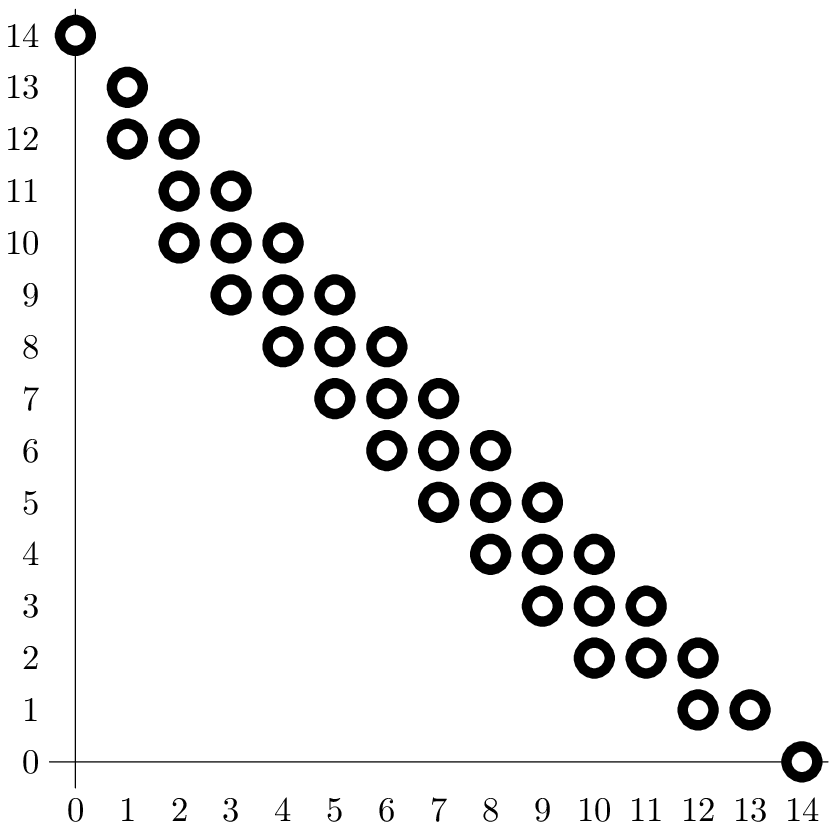}}\hfil

\subfloat[Using normal $Z_q$]{\includegraphics[scale=0.6]{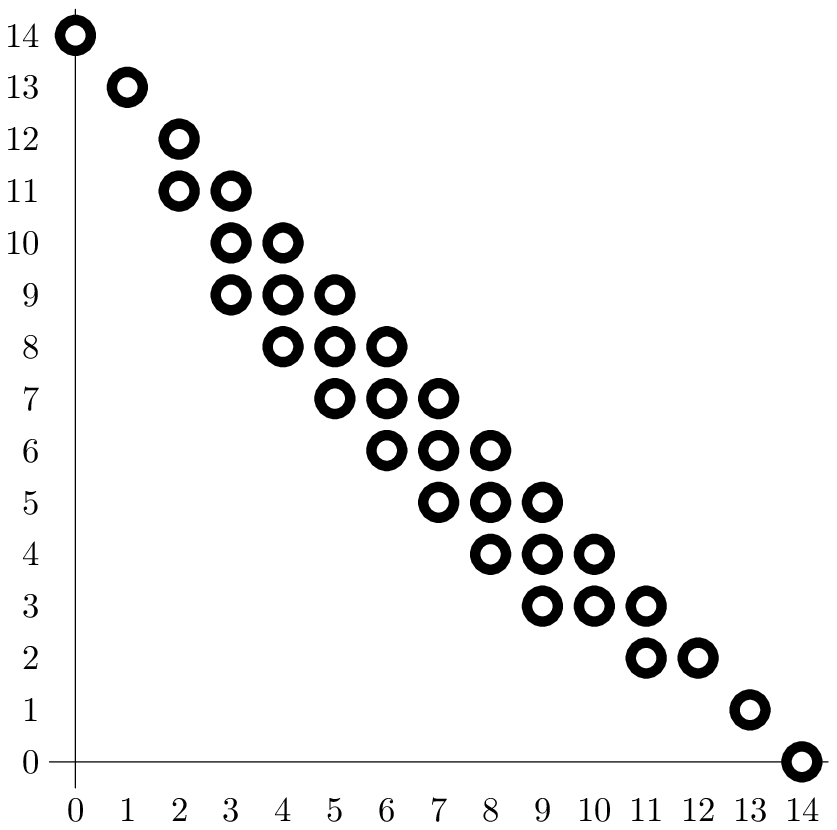}}\hfil
\subfloat[Using $\widehat{Z}_q$]{\includegraphics[scale=0.6]{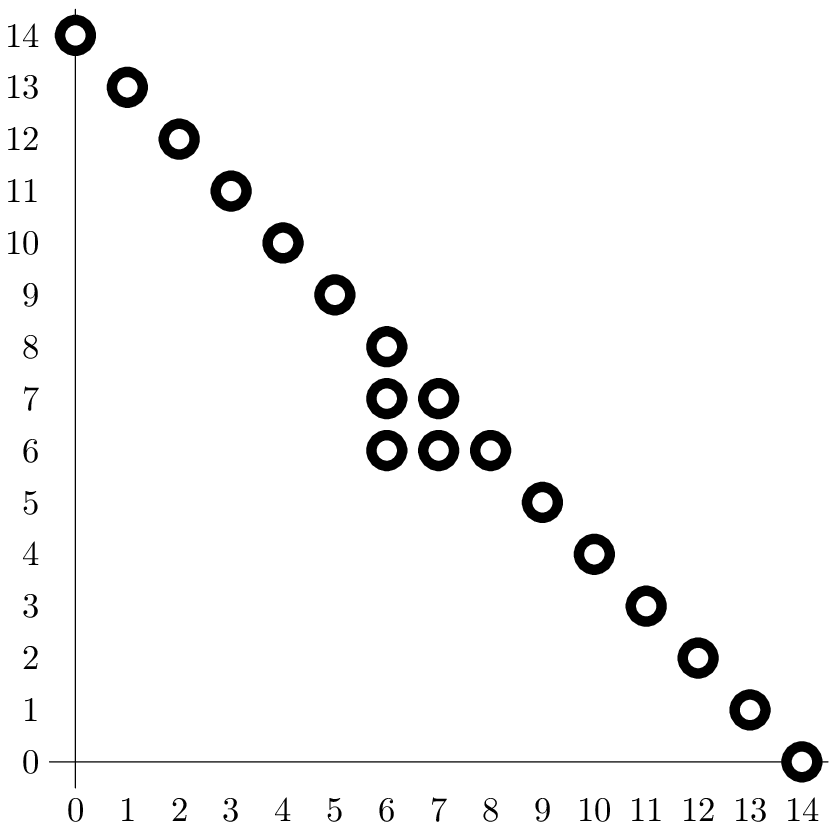}}\hfil
\caption{Lower bounds for the inertia set of Desargues graph}
\label{fig:desbnds}
\end{figure}

\begin{itemize}
\item[(a)] By Proposition~\ref{prop:AIM} we have that the maximum nullity of a matrix associated with the graph is bounded by $Z(G)$.  For Desargues graph $Z(G)=8$ and a forcing set is shown in Figure~\ref{fig:des}a.  This shows that we must have $p+q\ge 12$ giving us the set shown in Figure~\ref{fig:desbnds}a.
\item[(b)] In addition to the work done in (a) we have that by Theorem~\ref{thm:Z+} that the maximum nullity of a positive semidefinite matrix (i.e., $q=0$) is bounded by $Z_+(G)$.  For Desargues graph $Z_+(G)=6$ and a forcing set is shown in Figure~\ref{fig:des}c.  This shows that $(12,0)$ and $(13,0)$ cannot be in the inertia set, giving us the set shown in Figure~\ref{fig:desbnds}b.
\item[(c)] When we work with $Z_q$ we can insist that Black must always spend up front before forcing.  This is a bad choice on Black's part in general, but we can sometimes get useful information.  In the case of Desargues graph there is a set on $7$ vertices shown in Figure~\ref{fig:des}b which if Black initially spends on these vertices, then Black can win in the $Z_1$ game.  This shows that $(11,1)$ is not a point in the inertia set, giving us the set shown in Figure~\ref{fig:desbnds}c.
\item[(d)] When we use the definition for the $Z_q$ game (i.e., Black can save tokens to spend later), then we have
\[
Z_0(G)=Z_1(G)=6<Z_2(G)=7<Z_3=\cdots=8.
\]
This gives the set shown in Figure~\ref{fig:desbnds}d.
\item[(e)] When we apply $\widehat{Z}_q$ to Desargues graph we see a remarkable improvement.  In particular, we have
\[
Z_0(G)=\cdots=Z_5(G)=6<Z_6(G)=\cdots=8.
\]
This gives the set shown in Figure~\ref{fig:desbnds}e.
\end{itemize}

We have now computed lower bounds, the best coming from $\widehat{Z}_q$.  We have not yet determined the inertia set for Desargues graph, and it is not completely known.  However it is known that $(6,6)$ and $(14,0)$ are in the inertia set.  To construct the point $(6,6)$ we construct a $0,1,-1$ matrix in ${\cal S}(G)$ by puttings $0$s on the diagonal and $1$ for the edges except for five edges which receive $-1$, these five edges are every other spoke between the inner and outer cycles and are marked in Figure~\ref{fig:desmarked}.

\begin{figure}[hbt]
\centering
\includegraphics[scale=1]{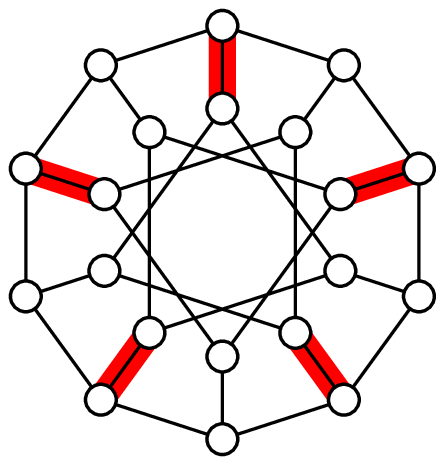}
\caption{Black edges are assigned $1$, red edges are $-1$ to get $(6,6)$ for the inertia set.}
\label{fig:desmarked}
\end{figure}

The eigenvalues for the resulting matrix are $[-\sqrt5]^6,[0]^8,[\sqrt5]^6$ (where exponents denote multiplicity).  This shows that $(6,6)$ is in the set by construction.  Further if we add $\sqrt5I$ then the resulting matrix has eigenvalues $[0]^6,[\sqrt5]^8,[2\sqrt5]^6$ showing that $(14,0)$ is in the set.

\subsection*{Using $\widehat{Z}_q$ to show diagonals are zero}
For the matrix that we constructed to show that $(6,6)$ was in the inertia set for Desargues graph we had that the diagonal entries were all zero.  For Desargues graph it can be shown that the \emph{only} matrices that can give $(6,6)$ must have all zeroes on the diagonal.

This can be shown by looking at $\widehat{Z}_q$.  In particular, if we look over all the possible ways to have the vertices looped and unlooped then we see that $\widehat{Z}_6(G)$ achieves the maximum at the unique assignment of unlooped at each vertex, i.e., any matrix with a nonzero on the diagonal would not have maximum nullity (assuming that a matrix does achieve this maximum).

Desargues graph is not the only graph with this property.  Others include complete multipartite graphs where each part is of size at least $3$, as well as $K_{5,5}$ minus a perfect matching and many others.

\subsection*{Disjoint union of graphs}
When computing the inertia set of a disjoint union of graphs we can compute the inertia of each graph separately and then take the Minkowski sum of the inertia sets.  Similarly, given the lower bounds for two graphs we can take the Minkowski sum of their lower bounds to produce a lower bound for their disjoint union.

For finding the lower bounds of inertia sets by computing $Z_q$ of the disjoint union of two or more graphs, it is better to do it for each graph separately and combine the results.  One reason is that this is computationally easier.  Another reason is that for some graphs the bounds are worse when we compute $Z_q$ together than when we combine the two results.  For example consider the graph $Q$ shown in Figure~\ref{fig:unionbad}.

\begin{figure}[h]
\centering
\includegraphics{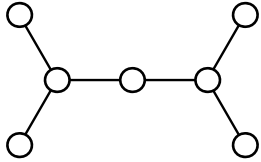}
\caption{A tree on $7$ vertices with $Z_0(Q)=1$, $Z_1(Q)=2$, and $Z_2(Q)=3$.}
\label{fig:unionbad} 
\end{figure}

By taking the Minkowski sum of two copies of the lower bounds for this graph we can conclude that for the disjoint union of two copies of $Q$ that the maximum nullity is at most $4$ when considering matrices with $2$ negative eigenvalues.  However we have $Z_2(Q\sqcup Q)=5$, which is a worse bound.

In general it is possible to show that
\[
\max_{s+t=q}\big(Z_s(G)+Z_t(H)\big)\le Z_q(G\sqcup H),
\]
and the example above shows that we can have a strict inequality.  (Note that if the parameter did behave like a Minkowski sum we would have an equality in the above relationship.)

\section{Algorithmic implementation for computing $Z_q$}\label{sec:algo}
Several algorithms for computing the zero forcing number of a graph have been implemented in {\sc sage}.  One brute force method is to simply examine all subsets of the vertices, starting with the singletons, and determine which of these subsets when colored black can force all of the other vertices to black by repeated application of the color change rule.  The smallest such subset is a minimal zero forcing set and its size gives the zero forcing number.

Our approach will be a similar brute force approach in that we will look at all sets $B\subseteq V$ and determine the minimal number of tokens that Black must spend to make all the vertices black given that Black is starting with the vertices in $B$ already colored black.  The algorithm will make use of the color change rule and we will continually apply the color change rule when we can since this is always free (this reduces the amount of information that needs to be stored and the number of cases that need to be considered).   Algorithm~\ref{alg:force} gives a simple method to force as much as possible given a subset of vertices of a graph already colored black.  (For a vertex $v$ of a graph, ${\rm nbd}(v)$ is the set of neighbors of $v$.)

\begin{algorithm}
\SetKwInOut{Input}{input}\SetKwInOut{Output}{output}\SetKw{Return}{return}
\Input{A graph $G=(V,E)$ with $B\subset V(G)$ a set of vertices colored black}
\Output{Maximal set of vertices $C$ that can be colored black by using repeated applications of the color change rule}
\BlankLine
$C\leftarrow B$\;
\Repeat{$B=C$}{
	$B\leftarrow C$\;
	\ForAll{$v\in B$}{
		\If{${\rm nbd}(v)\setminus B=\{w\}$}{
			$C\leftarrow C\cup \{w\}$\;
		}
	}
}
\Return{$C$\;}
\caption{The procedure $F(G,B)$ for applying the color change rule.}
\label{alg:force}
\end{algorithm}

The approach for finding $Z_q(G)$ will be to calculate the minimal number of tokens needed to guarantee that Black can get from a current coloring of vertices to having all vertices colored black under the rules of the game.  In other words, for each $U\subseteq V$ we associate a cost to it, denoted ${\rm cost}(U)$, that will indicate the minimal number of tokens Black needs to get all vertices colored black given that $U$ is already colored black.  So for example when all vertices are already colored black we need $0$ tokens and so ${\rm cost}(V)=0$ while when all vertices are colored white then the minimal cost is $Z_q(G)$, i.e., ${\rm cost}(\emptyset)=Z_q(G)$.

We work backwards, since we already know ${\rm cost}(V)$, from the large subsets to smaller subsets.  At each stage we consider each of the three options available to Black and choose the option which minimizes cost.  The most difficult part of doing this is determining what occurs for the option when White is involved (lines 5-12 of Algorithm~\ref{alg:Znum}).  In this case we consider every possible set of subsets to hand to White and we calculate the cost of each one being returned to us, since we are assuming that White will try to block us we take the cost as the most expensive of these options.

The steps involved are shown in Algorithm~\ref{alg:Znum}, a variation of which has been implemented in {\sc sage} \cite{sage}, and is publicly available online.

\begin{algorithm}
\SetKwInOut{Input}{input}\SetKwInOut{Output}{output}
\SetKwFunction{cost}{cost}\SetKw{Return}{return}
\SetKw{KwWith}{with}\SetKw{KwLet}{let}
\SetKw{KwAnd}{and}
\SetKwComment{tp}{//}{}
\Input{A graph $G=(V,E)$ and parameter $q$}
\Output{The value $Z_q(G)$}
\BlankLine
$\cost(V)\leftarrow 0$\;
\For{$i\leftarrow |V|-1$ \KwTo $0$}{
	\ForEach{$U\subseteq V$ \KwWith $|U|=i$ \KwAnd $F(G,U)=U$}{
		$b,c,\cost(U)\leftarrow\infty$\;
		\KwLet $K$ be the sets of vertices of the connected components of $G\setminus U$\;
		\ForEach{$J\subseteq K$ \KwWith $|J|=q+1$}{
			$b'\leftarrow -\infty$\;
			\ForEach{$I\subseteq J$ \KwWith $I\ne\emptyset$}{
				$b'=\max\big(b',\cost(F(G,F(G[U\cup I[,U)))\big)$\;
			}
			$b\leftarrow\min(b,b')$\;
		}
		\ForEach{$v\in V\setminus U$}{
			$c\leftarrow\min(c,\cost\big(F(G,U\cup \{v\})\big)+1)$\;
		}
		$\cost(U)\leftarrow\min(b,c)$\;
	}
}
\Return{$\cost(\emptyset)$\;}
\caption{The procedure $Z(G,q)$ to calculate the zero forcing numbers}
\label{alg:Znum}
\end{algorithm}

The cost function generated by this algorithm can be used by Black to determine a strategy for playing that uses at most $Z_q(G)$ tokens.  Namely at each stage Black chooses an option that will let them win with the number of tokens they have available to them.  Also, Black will be able to win by spending all the tokens up front if and only if there is some $U\subseteq V$ with $|U|=Z_q(G)$ and ${\rm cost}(U)=0$.

\subsection*{Acknowledgements}
The ideas for this note grew out of a problem session at the 2010 NSF-CBMS conference, ``The Mutually Beneficial Relationship of Matrices and Graphs'', held at Iowa State University and supported by grant DMS0938261.  We particularly thank Wayne Barrett, Leslie Hogben, Steven Osborne and John Sinkovic for useful discussions and ideas.

\end{document}